\documentclass[10pt]{article}

\usepackage{amsthm,amsmath,amssymb,amsbsy,amsfonts,graphics,epsfig,graphicx}

\usepackage{etoolbox}
\let\bbordermatrix\bordermatrix
\patchcmd{\bbordermatrix}{8.75}{4.75}{}{}
\patchcmd{\bbordermatrix}{\left(}{\left[}{}{}
\patchcmd{\bbordermatrix}{\right)}{\right]}{}{}

\newcommand{\R}{{\mathbb R}}
\newcommand{\C}{{\mathbb C}}

\newcommand{\Z}{{\mathbb Z}}
\newcommand{\N}{{\mathbb N}}
\newcommand{\D}{{ \mathbb{T}}}
\newcommand{\M}{\mathcal{M}}
\newcommand{\Oo}{{\cal O}}

\newcommand{\mbf}{\mathbf}

\newcommand{\beq}{\begin{equation}}
\newcommand{\eeq}{\end{equation}}
\newcommand{\beqa}{\begin{eqnarray}}
\newcommand{\eeqa}{\end{eqnarray}}
\newcommand{\bmat}{\begin{bmatrix}}
\newcommand{\emat}{\end{bmatrix}}

\newcommand{\tr}{\mbox{Tr}}

\newcommand{\dsp}{\displaystyle}

\newcommand{\vmv}{{\cal VMV}}
\newcommand{\KMS}{Kac-Murdock-Szeg\H{o} }

\newtheorem{theo}{{\bf Theorem}}[section]
\newtheorem*{theo*}{{\bf Theorem}}
\newtheorem{cor}[theo]{{\bf Corollary}}
\newtheorem{lem}[theo]{{\bf Lemma}}

\theoremstyle{definition}
\newtheorem{defn}{{\bf Definition}}[section]

\numberwithin{equation}{section}

\author{A. Bourget and T. McMillen \thanks{\textbf{Mailing address}: Department of Mathematics, California State University (Fullerton),
McCarthy Hall 154, Fullerton CA 92834 (US). \textbf{Email address}: abourget@fullerton.edu, tmcmillen@fullerton.edu }}

\title{A First Szeg\H{o}'s Limit Theorem for a class of non-Toeplitz matrices}

\begin{document}

\maketitle

\bibliographystyle{plain}

\begin{abstract}
We compute the limiting statistical distribution of the eigenvalues of sequences of matrices whose entries satisfy what we call a vanishing mean variation condition and are $\mu$-distributed for some probability measure.
As an application of our results, we extend the well-known class of \KMS generalized Toeplitz matrices to sequences of matrices whose diagonal entries are modeled by Riemann integrable functions.
\end{abstract}

\noindent \textbf{Keywords:} First Szeg\H{o}'s Limit Theorem, Kac-Murdock-Szeg\H{o} matrices, vanishing mean variation, equidistributed sequences.\\

\noindent \textbf{MSC}: 35P20, 47B35, 41A60, 15B05

\section{Introduction}

Kac, Murdock and Szeg\H{o} \cite{kamusz53} introduced  generalized Toeplitz matrices in 1953.  These are defined  in the following way.  Let $a(s,t)$ be a complex valued function (called the symbol of the matrix) such that the Fourier coefficients
\beq
\hat{a}_k(s) = \frac{1}{2\pi} \int_{-\pi}^{\pi} a(s,t) e^{-ikt} dt
\eeq
are defined on $[0,1]$.  For each integer $n$, define the $ n \times n$ matrix
\beq
T_n(a) = \left[\hat{a}_{j-i}\left(\frac{i+j}{2n+2}\right)\right]_{i,j=0}^{n-1}
\eeq
The following generalizes  Szeg\H{o}'s First Limit Theorem.
\begin{theo*}[Kac, Murdock and Szeg\H{o} (1953)]
\label{KMStheorem}
Suppose the following conditions hold.
\begin{enumerate}
\item[(i)]
The symbol $a(s,t)$ is real valued.
\item[(ii)]
The functions $\hat{a}_k(s)$ are continuous on $[0, 1]$.
\item[(iii)]
There exists a constant ${\cal N}$ such that
\begin{equation}
\label{cond szego 0}
{\cal N}:= \sum_{k=-\infty}^{\infty} \| \hat{a}_k \|_\infty < \infty.
\end{equation}
\end{enumerate}
Denote the eigenvalues of $T_n(a)$ by $\lambda_{k}(T_n(a))$.  Then $\lambda_{k}(T_n(a)) \in [-{\cal N}, {\cal N}]$, and
one has the following:
\begin{equation}
\lim_{n\rightarrow\infty} \frac{1}{n}\sum_{k=1}^n \varphi\left(\lambda_{k}(T_n(a)) \right)  = \frac{1}{2\pi}  \int_{-\pi}^{\pi} \int_0^1 \varphi\left(a(s,t)\right) ds\,dt
\label{2.1}
\end{equation}
for any $\varphi \in C([-{\cal N}, {\cal N}])$.
\end{theo*}
The above theorem says, roughly, that as $n\rightarrow\infty$, the eigenvalues of $T_n(a)$ distribute like the values of $a(s,t)$ sampled at regularly spaced points in the rectangle $[0, 1]\times [-\pi, \pi]$.  Thus, one obtains the limiting statistical distribution (LSD) of the eigenvalues of $T_n(a)$. In the case when the symbol does not depend on $s$, $a(s,t)=a(t)$,  \eqref{2.1} reduces to the First Szeg\H{o}'s Limit Theorem for Toeplitz matrices \cite{grsz58}.


The \KMS matrices are generalizations of Toeplitz matrices, in one sense, since as long as the functions $\hat{a}_k$ are continuous, the diagonals of $T_n(a)$ satisfy a small deviation condition.  A key point we want to make with this paper is that it is the small deviation condition that allows the calculation to go through.  Taking the diagonals from continuous functions, as Kac, Murdock, and Szeg\H{o} do, is only one way to ensure this condition.  (See \S\ref{illustrativeexample} for  examples of  matrix sequences that are not of the \KMS  type.)

\smallskip

The main purpose of this paper is to  generalize \eqref{2.1} to a larger class of operators. More precisely, we compute the LSD, defined as the weak-limit of the measures
$ \frac{1}{n} \sum_{k=1}^n \delta_ {\lambda_k(A_n)}$,
for the eigenvalues $\lambda_k(A_n)$ of sequences of matrices $\{A_n\}$ for which the entries along each of their diagonals satisfy what we call a vanishing mean variation condition (see Definition~\ref{ECN}) and are asymptotically distributed for some probability measure $\mu$ on some compact subspace $X$ in  $\C^{k}$ or $l^1(\C)$ (see Definition~\ref{mu distributed}).  In this case we have
\begin{equation}
\lim_{n\rightarrow\infty} \frac{1}{n} \sum_{k=1}^n \varphi(\lambda_k(A_n))  = \frac{1}{2\pi}  \int_{-\pi}^{\pi} \int_{X} \varphi\left(F(\pmb{z},t)\right) \,d\mu(\pmb{z})\, dt
\label{mainresult}
\end{equation}
where $F(\pmb{z},t)$ is the Fourier series defined in \eqref{F}. For Hermitian matrices $A_n$, the formula \eqref{mainresult} holds for compactly supported continuous functions $\varphi$ on $\R$, and hence gives the LSD of the sequence $\{A_n\}$.  For arbitrary matrices, the formula holds for analytic $\varphi$ in $\C$.


In addition to the small deviation condition, another important aspect of \eqref{mainresult} is the presence of the measure $\mu$. Because of its greater generality, we use \eqref{mainresult}  to rederive and extend a number of results.  For instance,  in Theorem \ref{cor main 2} below, we show that condition (ii) in the Kac-Murdock-Szeg\H{o} Theorem above may be replaced by the condition that  the $\hat{a}_k$ are Riemann integrable.


Our approach is based on the standard moments method. We begin by setting up some terminology and definitions in the next section. In the third section,  we start by computing the moments of sequences of matrices of fixed band size and then extend our results to the non-band case.  In \S\ref{illustrativeexample} we give some examples to illustrate the novelty of our approach. We conclude with a discussion of open questions.

\section{Notation and definitions }

All of the results below are greatly simplified if we number the entries of matrices along their diagonals rather than the usual row-column positions. For this reason,
we write the $n \times n$ matrix $A_n$ in the following form:
 \begin{equation}
A_n =[a_{k;j}]= \bmat a_{0;0} & a_{1;0} &  &  & a_{k;0} & & a_{n-1;0}  \\
 a_{-1;0}  & a_{0;1} & \ddots & & & \ddots &\\
 & a_{-1;1} & \ddots &   &  &    & a_{k;n-k}\\
& & \ddots &  & \ddots\\
a_{-k;0} &&&  & \ddots & a_{1;n-2} \\
&\ddots &&& \ddots &a_{0;n-1}  & a_{1;n-2} \\
a_{-n+1;0} &   & a_{-k;n-k} & & & a_{-1;n-2} & a_{0;n-1}
\emat.
\label{diagnotation}
\end{equation}
In other words, $a_{k;j}=a_{k;j}(n)$ denotes the $j$th term on the $k$th diagonal, with $k=0$ corresponding to the main diagonal.  A Toeplitz matrix has $a_{k:j} = a_k$.  The entries may depend on the size $n$ of the matrix, but we will usually suppress the dependence of $a_{k;j} = a_{k;j}(n)$ on $n$ where this is clear.
A matrix with band size $k_0$ is zero outside the band $k_0$ from the main diagonal, i.e. $a_{k;j} = 0$ for $|k|>k_0$.

%

\smallskip

We denote the eigenvalues of $A_n$ by $\lambda_1(A_n),...,\lambda_n(A_n)$ and its singular values by $\sigma_1(A_n) \geq \dots  \geq \sigma_n(A_n)$. The trace norm of $A_n$ is defined by
\begin{equation} \label{trace norm}
 \|A\|_{tr}=\|A\|_1 = \sum_{k=1}^n \sigma_k(A).
 \end{equation}
On a few occasions, we also need to consider the spectral norm defined by
\begin{equation}
\|A\|_\infty = \sigma_1(A).
\end{equation}

\smallskip

We denote by $l^1=l^1(\C)$, the space of all complex sequences $\{z_k\}_{k \in \Z}$ for which $\sum_{k \in \Z} |z_k|$ is finite. We also denote throughout the paper the Fourier series
with coefficients $\pmb{z}=\{z_k\}_{k \in \Z} \in l^1$ by
\begin{equation} \label{F}
F(\pmb{z},t) : = \sum_{k \in \Z} z_k e^{ikt} \qquad (t \in [-\pi,\pi]).
\end{equation}
Note that $F(\pmb{z},\cdot) \in \mathcal{A}(\D)$, the Wiener algebra of summable Fourier series. When $z_k=0$ for $|k|>k_0$, we  write $\pmb{z}=(z_{-k_0},...,z_{k_0})$ and let
\begin{equation} \label{P}
P(\pmb{z},t) : = \sum_{|k| < k_0} z_k e^{ikt} \qquad (t \in [-\pi,\pi])
\end{equation}
be the trigonometric polynomial with coefficients $\pmb{z} \in \C^{2k_0+1}$.  

\subsection{Vanishing mean variation sequences}

In his seminal work on the spectral  theory of Jacobi matrices, Simon \cite{si09} introduced the Ces\`aro-Nevai class as the set of Jacobi matrices $J(\pmb{a},\pmb{b})$ whose sequences $\pmb{a}=\{a_k\}$ and $\pmb{b}=\{b_k\}$ satisfy the condition
\begin{equation} \label{SCN}
    \sum_{k=1}^n |a_{k}| + |b_k-1| =o(n).
\end{equation}
His definition was motivated by the study of Jacobi matrices that are perturbations of the free discrete Schr\"odinger operator $J(\pmb{0},\pmb{1})$.  In particular, if \eqref{SCN} holds, then $J(\pmb{a},\pmb{b})$ has the same LSD as $J(\pmb{0},\pmb{1})$. Following \eqref{SCN}, the first author considered in \cite{bo12} sequences of  Jacobi matrices $\{J_n\}$ that satisfy
\begin{equation} \label{SCN2}
    \sum_{k=1}^n |a_{k+1}(n)-a_k(n)| + |b_{k+1}(n)-b_k(n)| =o(n).
\end{equation}
In the definition below, we extend \eqref{SCN2} to  sequences of general matrices.

\begin{defn}  \label{ECN}
We say that a matrix sequence $\{A_n\}$ is of vanishing mean variation if the entries along the diagonals of $A_n$ satisfy the asymptotic condition
\begin{equation} \label{vmv cond}
\sum_{j=0}^{n-k-1} \left| a_{k;j+1} (n) -a_{k;j} (n) \right| = o(n)
\end{equation}
for each $k$.
We denote by $\mathcal{VMV}$ the set of all such matrix sequences.
\end{defn}

\smallskip

In particular, the vanishing mean variation condition allows us to shift the indices of products of entries of $A_n$. This observation will play a crucial role in the proof of Theorem \ref{main trace} below. Here are some basic examples of sequences that belong to $\mathcal{VMV}$ and that we consider later on:

\begin{itemize}
\item[ (i)] If $T(a)$ is a Toeplitz operator with symbol $a \in \mathcal{A}(\D)$, then the sequence $\{T_n(a)\}$ is obviously in $\mathcal{VMV}$.

\smallskip

\item[(ii)] More generally, any sequence $\{A_n\}$ with diagonals entries given by density one convergent sequences, i.e. for every $\varepsilon>0$,
$$ \#\{j \leq n: |a_{k;j}(n) - a_k|>\varepsilon \} =o(n)$$
is easily seen to belong to $\mathcal{VMV}$.

\smallskip

\item[(iii)] Results for the \KMS matrices can be extended to sequences $\{A_n\}$ whose diagonals are modeled by Riemann integrable functions.  Indeed, let $P_n=\{t_{0;n},t_{1;n},...,t_{n;n}\}$ be a sequence of partitions of $[0,1]$ with $\text{mesh}(P_n)=o(1)$ and let $a_{k;j}(n)=\hat{a}_k(t_{j;n})$ for some Riemann integrable functions $\hat{a}_k$ on $[0,1]$.  It is straightforward to prove that $\{A_n\} \in \mathcal{VMV}$.
\end{itemize}

\smallskip

\subsection{$\mu-$distributed sequences}

Our second definition is based on the standard notion of asymptotically distributed sequences with respect to a probability measure $\mu$ on compact metric spaces $X$.  A sequence $\{x_k\}$ in $X$ is said to be asymptotically distributed wrt $\mu$ if the measures $ \frac{1}{n} \sum_1^n \delta_{x_k}$ converge weakly to $\mu$. We extend this notion to sequences of matrices in the following manner.

\smallskip

\begin{defn}
\label{mu distributed}
 Let $X$ be a compact subspace of $l^1$. The sequence of matrices $\{A_n\}$ is said to be $\mu$-distributed for some Borel probability measure $\mu$ on $X$ if
$\frac{1}{n} \sum_{j=0}^{n} \delta_{\pmb{a}_j(n)}$
 converge weakly to $\mu$ with
 \begin{equation} \label{ajn}
 \pmb{a}_j(n)=\{...,0,a_{-j;j},a_{-j+1;j}...,a_{n-j;j},0,...\}.
 \end{equation}
 For $\{A_n\}$ of fixed band size $k_0$, we write $\pmb{a}_j(n)=(a_{-k_0;j},...,a_{k_0;j})$ and take $X$ to a be compact subspace in $\C^{2k_0+1}$.
\end{defn}

\smallskip

All of the examples given in the previous section are $\mu$-distributed. Indeed, we have:

\begin{itemize}
 \item[(i)] The sequence $\{T_n(a)\}$ obtained from a Toeplitz matrix $T(a)$ with $a \in \mathcal{A}(\D)$ is $\delta_{\pmb{a}}$-distributed on $l^1$ with $\pmb{a}=\{a_k\} \in l^1$.

\smallskip

\item[(ii)] If $\{A_n\}$ has density one convergent diagonal sequences as in (ii) above
to $\pmb{a}=\{a_k\} \in l^1$, then $\{A_n\}$ is  $\delta_{\pmb{a}}-$distributed.

\smallskip

\item[(iii)] Let $\{A_n\}$ be as in (iii) above and let $\pmb{\alpha}:[0,1] \to l^1$ be the map defined by
\begin{equation} \label{alpha map}
\pmb{\alpha}(t) = \{\hat{a}_k(t)\}.
\end{equation}
If for all $t$, $\pmb{\alpha}(t) \in X$ for some compact $X \subset l^1$, then $\{A_n\}$ is $m_{\pmb{\alpha}}$-distributed with $m_{\pmb{\alpha}}$ the push-forward of the Lebesgue measure $m$ under the map $\pmb{\alpha}$ .

\end{itemize}

\section{Main results}

We now come to the main results of the paper, i.e. Szeg\H{o}'s Limit Theorems for $\mathcal{VMV}$ sequences of matrices that are $\mu$-distributed. We begin by computing the moments for sequences of matrices $\{A_n\} \in \mathcal{VMV}$ of fixed band-size, then we extend our result to sequences of arbitrary band sizes.

\smallskip

\subsection{Sequences of band matrices}

We now consider sequences $\{A_n\}$ of fixed band-size $k_0$, i.e. $a_{k;j}(n)=0$ for all $|k| > k_0$. Throughout this section, we make the assumption that the entries of $\{A_n\}$ are uniformly bounded, i.e.
\begin{equation}
  \sup_n ( \max_{j,k} |a_{k;j}(n)|) < \infty.
\end{equation}
The proof of the following theorem is a modification of that of \KMS \cite{kamusz53,grsz58}, modified for the $\mu-$distributed case, and couched in terms of the notation of \eqref{diagnotation}.

\begin{theo} \label{main trace}
Let $\{A_n\} \in \mathcal{VMV}$ of fixed band size $k_0$ be $\mu$-distributed on some compact $X \subset \C^{2k_0+1}$.  Then for any $r,s \in \N$, we have
\begin{equation} \label{Trace 2}
    \lim_{n \to \infty} \frac{1}{n} \tr\left[A_n^r(A_n^*)^s) \right] = \frac{1}{2\pi}  \int_{-\pi}^{\pi} \int_{X}  P^r( \pmb{z} , t) \overline{P^s(\pmb{z},t)} \, d\mu(\mbf{z})  \, dt.
\end{equation}
\end{theo}

\smallskip

The proof of this theorem relies on the fact that we can shift indices modulo an $o(n)$ term.  This is the content of the following lemma.  In what follows we define $a_{k;j} = a_{k;j}(n)$ to be zero if  $j<0$, $j>n-k$, or $|k|>n$.
\begin{lem}
\label{shiftlemma}
Let $\{A_n\} \in \vmv$.  Then for any integers $\nu_1, \nu_2, \dots \nu_p$ and $h_1, h_2, \dots, h_p$,
\[ \sum_{j=0}^n a_{h_1; \nu_1 + j}a_{h_2;\nu_2+j}\cdots a_{h_p;\nu_p+j}
=  \sum_{j=0}^n a_{h_1;  j}a_{h_2; j}\cdots a_{h_p;j}
+o(n)
\]
\end{lem}
\begin{proof}
Shift the index in the first term:

\begin{eqnarray*}
\sum_{j=0}^n a_{h_1; \nu_1 + j}a_{h_2;\nu_2+j}\cdots a_{h_p;\nu_p+j}
&=& \sum_{j=0}^n a_{h_1;  j}a_{h_2;\nu_2+j}\cdots a_{h_p;\nu_p+j}
\\
&&
+ \sum_{j=0}^n \left(a_{h_1; \nu_1 + j}-a_{h_1; j}\right)a_{h_2;\nu_2+j}\cdots a_{h_p;\nu_p+j}
\end{eqnarray*}
Since the entries of $A_n$ are uniformly bounded, we can assume that they are bounded by 1.  Thus the error from shifting the first term is
\begin{eqnarray*}
\left|  \sum_{j=0}^n \left(a_{h_1; \nu_1 + j}-a_{h_1; j}\right)a_{h_2;\nu_2+j}\cdots a_{h_p;\nu_p+j} \right|
& \leq & \sum_{j=0}^n \left| a_{h_1; \nu_1 + j}-a_{h_1; j}\right|
\\
&\leq & |\nu_1| \sum_{j=0}^n \left| a_{h_1;  j+1}-a_{h_1; j}\right|
\\
&=&  o(n)
\end{eqnarray*}
by the triangle inequality.  Shifting the indices in the remaining terms similarly will result in at most an $o(n)$ error.
\end{proof}

With this lemma we can proceed to the proof of Theorem~\ref{main trace}.

\begin{proof}[Proof of Theorem~\ref{main trace}]
As in \cite{kamusz53},
we write $A_n$ as the sum of  diagonals:
\begin{equation} \label{diag}
A_n = \sum_{|k| \leq k_0}  D_k ,
\end{equation}
where $D_{k}$ denotes the $k$th diagonal matrix of $A_n$, i.e.
\[
D_{k} = \bmat
 & & & &  a_{k;0} \cr
& & & & &  a_{k;1} \cr
& & & & & &  \ddots \cr
& & & &  & & & a_{k:n-k} \cr
& \cr
& \cr
\emat
\]
Thus, $A_n^r (A_n^*)^s$ is the sum
$$ \sum  \prod_{j=1}^r D_{h_j}  \prod_{l=1}^s \overline{D}_{-k_l}.$$
where the sum is taken over all possible combinations. The main diagonal entries of each product in the last sum will  be nonzero only if
$$|h|-|k|= \sum h_j - \sum k_j = 0.$$

Using the Kac, Murdock and Szeg\H{o} \cite{kamusz53} approach,  we can calculate the entries on the main diagonal of the product $\prod_{j=1}^r D_{h_j}  \prod_{l=1}^s \overline{D}_{-k_l}$.   For $|h|= |k|$, the $i$th term on the diagonal of the product is
\beq
\left( \prod_{j=1}^r D_{h_j}  \prod_{l=1}^s \overline{D}_{-k_l} \right)_{0;i} =  \prod_{j=1}^r  a_{h_j; \nu_j + i} \prod_{l=1}^s \overline{a}_{-k_l;\nu_{p+l}+i}
\eeq
where
\begin{equation*}
   \nu_j = \begin{cases}
                   h_1+\cdots+h_{j-1}+ h_j^- & \text{ if } 1 \leq j \leq p \\
                   h_p - k_1 - \cdots - k_j^- & \text{ if } p< j \leq q.
            \end{cases}
 \end{equation*}
and $h_j^-=\min\{h_j,0\}$. Therefore, we obtain
\beq
\tr \left[ A_n^r (A_n^*)^s\right]  =  \sum_{j=0}^{n}  \sum_{|h|=|k|} \prod_{j=1}^r a_{h_j; \nu_j + i} \prod_{l=1}^s \overline{a}_{-k_l;\nu_{r+l}+i}
\label{trace}
\eeq

Since $\{A_n\}$ is banded, the sum $ \sum_{|h|=|k|}$ above is finite.
Thus we can use Lemma~\ref{shiftlemma}   to shift the indices modulo an $o(n)$ term:
\beq \label{trace2}
 \text{Tr}[A_n^r(A^*_n)^s] = \sum_{j=0}^n \sum_{|h|=|k|} \prod_{l=1}^r a_{h_l; j} \prod_{m=1}^s \overline{a_{-k_m;j}}  + o(n).
\eeq

Finally, it follows from the $\mu$-distribution of $\{A_n\}$ that
\begin{eqnarray*}
 \frac{1}{n} \text{Tr}[A_n^r(A^*_n)^s] & = & \int_X \sum_{|h|=|k|} \prod_{l=1}^r z_{h_l} \prod_{m=1}^s \bar{z}_{k_m} \, d\mu(\pmb{z}) + o(1) \\
  & = & \frac{1}{2\pi} \int_{-\pi}^\pi \int_X \prod_{m=1}^s \bar{z}_{k_m} \, d\mu(\pmb{z}) e^{i(|h|-|k|)} \, d\mu(\pmb{z}) \, dt + o(1)\\
  & = & \frac{1}{2\pi} \int_{-\pi}^\pi \int_X P^r(\pmb{z},t) \, \bar{P}^s(\pmb{z},t) \, d\mu(\pmb{z}) \, dt +o(1)
\end{eqnarray*}
as desired.
\end{proof}

For sequences of normal band matrices, one can use the Stone-Weierstrass Theorem and the functional calculus for normal operators together with our previous trace formula to obtain their  LSD.   More precisely, we have the following result.

\begin{cor}
Let  $\{A_n\}$ be a sequence of normal matrices of fixed band size $k_0$. If $\{A_n\} \in \mathcal{VMV}$ is $\mu$-distributed on some compact $X \subset \C^{2k_0+1}$, then  we have
\begin{equation} \label{Trace 4}
   \lim_{n \to \infty} \frac{1}{n} \tr\left[ \varphi(A_n)  \right] = \frac{1}{2\pi}  \int_{-\pi}^{\pi} \int_{X}  \varphi(P( \pmb{z} , t))\, d\mu(\mbf{z})  \, dt.
\end{equation}
for any $\varphi \in C_c(\C)$, the space of compactly supported continuous functions on $\C$.
\end{cor}

\medskip


%
%
%
\subsection{Sequences of non-band size matrices}

We now extend our previous results to sequences of matrices of arbitrary band size. We start with a basic lemma that gives a general condition for when two sequences have the same LSD.

\smallskip

\begin{lem}
\label{von neumann}
Let $\{A_n\}$ and $\{B_n\}$ be two sequences of matrices that satisfy $\|A_n\|_\infty =\Oo(1)=\|B_n\|_\infty$. If $\|A_n - B_n\|_{tr}=o(n)$, then we have
$$ \left| \sum_{k=1}^n \varphi(\lambda_k(A_n)) - \sum_{k=1}^n \varphi(\lambda_k(B_n)) \right| =o(n)$$
for any analytic function $\varphi$ on $D_\M$, i.e. $\varphi  \in C^\omega(D_\M)$. If the sequences are Hermitian, then we can choose $\varphi \in C(D_\M)$.
\end{lem}

\begin{proof}
By linearity of the trace and Mergelyan's Theorem, it suffices to consider $\varphi(z)=z^m$ for $m \in \N$.  Writing $A_n^m = ((A_n-B_n)+B_n)^m$ and using the elementary properties of the trace, we have
$$ \text{Tr}[A_n^m] - \text{Tr}[B_n^m] = \text{Tr}[(A_n-B_n) \, C_n] $$
for some matrix  $C_n$ that is a finite sum of products of $A_n-B_n$ and $B_n$. By the sub-multiplicative property of the the spectral norm, we easily deduce that $\|C_n\|_\infty=\Oo(1)$. Finally, we apply Von Neumann's trace inequality (see \cite{hj85}, Theorem 7.4.10, p. 433) to obtain
\begin{eqnarray*}
\left| \text{Tr}[(A_n-B_n) \, C_n] \right|   \leq   \sum_{k=1}^n \sigma_k(A_n-B_n) \, \sigma_k(C_n) =\Oo( \| A_n-B_n\|_{tr} ) =o(n)
\end{eqnarray*}
as desired.
\end{proof}

\smallskip

The basic idea underlying the proof of the results below is to approximate the LSD of the original matrix sequence by the LSD of another sequence of band-matrices with fixed band size. Consequently, we make the assumption similar to that of  Kac, Murdock and Szeg\H{o} \eqref{cond szego 0} that $\{A_n\}$ satisfies the condition
\begin{equation} \label{cond szego}
\M :=  \sup_n \left[ \sum_{|k| \leq n} \max_{0 \leq j \leq n-k} |a_{k;j}(n)| \right] <\infty.
\end{equation}
From Gershgorin's Circle Theorem \cite{hj85}, the spectrum of $A_n$ lies inside the closed disk $D_{\M}=\{z:|z| \leq \M\}$. Moreover, by Qi's extension of Gergorin's Theorem \cite{qi84}, the singular values of $A_n$ are also contained  $[0,\M]$.

\smallskip

\begin{theo}  \label{main normal}
Let $\{A_n\} \in \mathcal{VMV}$ be a $\mu$-distributed on some compact $X \subset l^1$. Then, we have
\begin{equation*}
   \lim_{n \to \infty} \frac{1}{n} \sum_{k=1}^n  \varphi \left(\lambda_k( A_n) \right) = \frac{1}{2\pi}\int_{-\pi}^{\pi} \int_{X} \varphi(F(\pmb{z},t))  \, d\mu(\pmb{z}) \, dt
\end{equation*}
for any $\varphi \in C^\omega(D_{\M})$. In addition,  if the $A_n$'s are assumed to be Hermitian, then we can take $\varphi \in C([-\M,\M])$.
\end{theo}

\begin{proof}   For any $\varepsilon>0$, condition \eqref{cond szego} implies there exists $k_0 \in \N$ such that
\begin{equation} \label{main normal 1}
 \sum_{|k|>k_0} \max_{1 \leq j \leq n-k} |a_{k;j}|   < \varepsilon.
\end{equation}
Let $\{B_n\}$ be the sequence of band-matrices of size $k_0$ obtained from $A_n$, i.e.  $B_n=[b_{k;j}]$ with
$$ b_{k;j}= \begin{cases}
                       a_{k;j} & \text{ if } |k| \leq k_0\\
                       0 & \text{otherwise}
                   \end{cases}$$
In particular, by Qi's result we deduce $\|A_n-B_n\|_{tr} = \Oo(n \varepsilon)$.  Thus, by  Lemma~\ref{von neumann}, it suffices to prove the result for $\{B_n\}$. But this follows from our trace formula in Theorem \ref{main trace}. Indeed, if $\pi_{k_0}:l^1 \to C^{2k_0+1}$ denotes the standard projection defined by
$$ \pi_{k_0} ( \pmb{z} ) = (z_{-k_0},...,z_{k_0})$$
then the sequence $\{B_n\}$ is $(\pi_{k_0})_*(\mu)$-distributed on the compact set $\pi_{k_0}(X) \subset \C^{2k_0+1}$. Hence,  we have
\begin{equation} \label{normal 2}
 \frac{1}{n}\text{Tr}[ B_n^m ]  = \frac{1}{2\pi}  \int_{-\pi}^{\pi} \int_{X}  F^m(\pi_{k_0}(\pmb{z}),t)  \, d\mu(\pmb{z}) \, dt + o(1).
\end{equation}
Moreover, by the compactness of $X$, we can choose $k_0$ large enough so that
$$| F(\pmb{z},t) - F(\pi_{k_0}(\pmb{z}),t)| < \varepsilon $$
uniformly in $\pmb{z}$ and $t$. Consequently, it follows
\begin{eqnarray*}
   \frac{1}{n} \text{Tr}[A_n^m] & = & \frac{1}{n} \text{Tr}[B_n^m] + \Oo(\varepsilon) \\
      & = & \frac{1}{2\pi} \int_{-\pi}^{\pi} \int_X F^m(\pmb{z};t) \, d\mu(\pmb{z}) \, dt + \Oo(\varepsilon)+o(1).
 \end{eqnarray*}
The conclusion is then an immediate consequence of the linearity of the trace and Mergelyan's Theorem.
 \end{proof}

\smallskip

As a first consequence of the above result, we extend the class of \KMS sequences of matrices  to sequences whose diagonals are modeled by Riemann integrable functions.

 \begin{theo}
 \label{cor main 2}
Let $\{A_n\}$ be a sequence of matrices such that  there exists a sequence of partitions $\{\mathcal{P}_n\}$ of $[0,1]$ with $\mathcal{P}_n=\{t_{j;n}\}_{j=0}^n$ and $mesh(\mathcal{P}_n)=o(1)$ such that
\[ a_{k;j}(n) = \hat{a}_k(t_{j;n}),
\]
where  $\hat{a}_k$ are Riemann integrable functions on $[0,1]$ satisfying
\begin{equation} \label{n-cond}
 \mathcal{N} := \sum_{k=0}^\infty \| \hat{a}_k \|_\infty < \infty.
\end{equation}
For any $\varphi \in C^\omega(D_{\mathcal{N}})$, we have
\begin{equation}
   \lim_{n \to \infty} \frac{1}{n}  \sum_{k=1}^n \varphi(\lambda_k(A_n)) =  \frac{1}{2\pi} \int_{-\pi}^{\pi} \int_0^1 \varphi( a(s,t))  \, ds \, dt
\end{equation}
where $a(s,t)=\sum_{k \in \Z} \hat{a}_k(s) e^{ikt}$. If the $A_n$ are Hermitian, then one can take $\varphi \in C([-\mathcal{N},\mathcal{N}])$.
\end{theo}

\begin{proof}
From the integrability of the $\hat{a}_k$, we have $\{{A}_n\} \in \mathcal{VMV}$. Let $\pmb{\alpha}:[0,1] \to l^1$ be the map given by \eqref{alpha map}.
Under condition \eqref{n-cond}, it is not hard to verify that the closure of the set $\pmb{\alpha}([0,1])$ is compact in $l^1$. By example (iii) of Section 2.2, the sequence
\begin{equation} \label{seq g}
  \hat{\pmb{a}}_j(n)=\{ ...,0,\hat{a}_{-j}(t_{j;n}),\hat{a}_{-j+1}(t_{j;n}),\dots ,\hat{a}_{n-j}(t_{j;n}),0,... \}
\end{equation}
is asymptotically equidistributed for the push-forward measure $m_{\pmb{\alpha}}$.  Consequently, Theorem~\ref{main normal} implies that
\begin{eqnarray*}
   \lim_{n \to \infty} \frac{1}{n} \tr[ A_n^m]
      & = & \frac{1}{2\pi} \int_{-\pi}^{\pi}  \int_{X}  F^m(\mathbf{z},t)  \, dm_{\pmb{\alpha}}(\mathbf{z}) \, d t\\
      & = & \frac{1}{2\pi} \int_{-\pi}^{\pi} \int_0^1 F^m(\pmb{\alpha}(s),t) \, ds \, d t \\
      & = & \frac{1}{2\pi} \int_{-\pi}^{\pi} \int_0^1a^m(s,t) \, ds \, d t\\
\end{eqnarray*}
as desired.
\end{proof}

\smallskip

The results of the above theorem also hold if the entries of $A_n$ are asymptotically modeled by Riemann integrable functions.

\begin{cor} Let $\{A_n\}$ be a sequence of matrices that satisfies \eqref{cond szego} and let $\{\hat{A}_n\}$ be a sequence as in Theorem \ref{cor main 2}. If we assume
\begin{equation} \label{A1}
|a_{k;j} - \hat{a}_k(t_{j;n})|=o(1),
\end{equation}
then we have
\begin{equation}
   \lim_{n \to \infty} \frac{1}{n}  \sum_{k=1}^n \varphi(\lambda_k(A_n)) =  \frac{1}{2\pi} \int_{-\pi}^{\pi} \int_0^1 \varphi( a(s,t))  \, ds \, dt
\end{equation}
for any $\varphi \in C^\omega(D_{\M'})$ with $\M'=\max\{\M,\mathcal{N}\}$. In the Hermitian case, $C^\omega(D_{\M'})$ is replaced by $C([-\M',\M'])$.
\end{cor}

\begin{proof}
Arguing as in the proof of Theorem \ref{main normal}, we can use Lemma \ref{von neumann} together with conditions \eqref{cond szego} and \eqref{n-cond} to reduce the problem to the case when $\{A_n\}$ and $\{\hat{A}_n\}$ are sequences of band matrices of fixed band size $k_0$. Under the assumption \eqref{A1} and Qi's Theorem, one has $\|A_n - \hat{A}_n \|_{tr} =o(n)$, so another application of Lemma \ref{von neumann} yields
\begin{equation}
 \frac{1}{n} \tr[\varphi(A_n)]  = \frac{1}{n} \tr[\varphi(\hat{A}_n)] +o(1)
 \end{equation}
 Moreover, by Corollary \ref{cor main 2}, we also have
 \begin{equation}
  \frac{1}{n} \sum_{k=1}^n \varphi(\lambda_k(\hat{A}_n))  =   \frac{1}{2\pi} \int_{-\pi}^{\pi} \int_0^1 \varphi( a(s,t))  \, ds \, dt +o(1).
 \end{equation}
The conclusion is then an immediate consequence of last two estimates.
\end{proof}


We conclude this section with two applications of the above results. First, we use Theorem \ref{main normal} to give a new proof of Szeg\H{o}'s First Limit Theorem. We state Szeg\H{o}'s Theorem in its most general form, i.e. for real-valued symbols $a \in L^1(\D)$ as considered by   Tyrtyshnikov and Zamarashkin \cite{TyZa98}.

\begin{cor}  (First Szeg\H{o}'s Limit Theorem)
Let $T(a)$ be a Toeplitz matrix with real valued symbol $a \in L^1(\D)$.  Then, we have
\begin{equation}
   \lim_{n \to \infty} \frac{1}{n}  \sum_{k=1}^n \varphi(\lambda_k(A_n))
   = \frac{1}{2\pi} \int_{-\pi}^{\pi} \varphi(a(t))   \, dt
\label{cor42}
\end{equation}
for any $\varphi \in C_c(\R)$, the space of compactly  supported functions on $\R$.
\end{cor}

\begin{proof}  First, note that $T_n(a)$ is Hermitian since $a$ is real valued. For $a>0$, so $T_n(a)$ is positive semi-definite, we have
$$ \| T_n(a)\|_{tr} = \sum_{k=1}^n \lambda_k(T_n(a)) = n a_0 = \frac{n}{2\pi} \int_{-\pi}^{\pi} a(t)  \ dt = \frac{n}{2\pi} \|a\|_1 .$$
By writing  $a=a^+ - a^-$ with $a^+(t)=\max\{a(t),0\}$ and $a^-(t)=-\min\{a(t),0\}$, it follows that $ \| T_n(a) \|_{tr} \leq  \pi^{-1} n \, \|a\|_1$ for any $a \in L^1(\mathbb{T})$. By the density of $C^1_c(\R)$ in $C_c(\R)$ for the sup norm, we only need to consider $\varphi \in C^1_c(\R)$. By the Mean-Value Theorem and the $p$-Wielandt-Hoffman inequality \cite{hj85} with $p=1$, we have
\begin{eqnarray*}
\sum_{k=1}^n \left| \varphi(\lambda_k(T_n(a))) - \varphi(\lambda_k(T_n(b))) \right| & \leq &  \|\varphi'\|_\infty \ \|T_n(a) - T_n(b) \|_{tr} \\
 & = & \| \varphi \|_\infty \  \| T_n(a-b) \|_{tr} \\
 & \leq &  \frac{n}{\pi} \, \|\varphi'\|_\infty \  \|a-b\|_1
\end{eqnarray*}
for $a,b \in L^1(\mathbb{T})$.  Hence, it suffices to consider $a$ in a dense subset of $L^1(\D)$, e.g. $\mathcal{A}(\D)$

\smallskip

For such  $a$, it is readily seen that condition \eqref{cond szego} holds with $\mathcal{M}=\|a\|_1$. Moreover, the sequence $\{T_n(a)\}$ is $\mu$-distributed in $X=\{\pmb{a}\}$ with $\mu=\delta_{\pmb{a}}$, and $\pmb{a}=\{a_k\}$ the sequence made by the Fourier coefficients of $a$. The conclusion is then an immediate consequence of  Theorem \ref{main normal}. The last part is a consequence that $T_n(a)$ is Hermitian if $a$ is real-valued.
\end{proof}

\smallskip


As a second application, we present a natural  extension of the well-known $M(a,b)$-class of Jacobi matrices introduced by Nevai \cite{Ne79}. Recall, a Jacobi matrix $J(\pmb{a},\pmb{b}) \in M(a,b)$ if it has convergent diagonals, i.e.
$$ \lim_{k \to \infty} a_{k} = a \qquad \text{and} \qquad \lim_{k \to \infty} b_k=b >0.$$
The LSD of such matrices is well-known to be the arcsine distribution over the interval $[a-2b,a+2b]$. In the next result, we extend Nevai's class to sequences of matrices $\{A_n\}$ whose diagonals are given by density one convergent sequences.

\begin{cor}
\label{corN}
Let $\{A_n\}$ satisfy, for each $\epsilon > 0$,
$$\#\{ j :  \| \pmb{a}_j(n) - \pmb{a} \|_{l^1} > \epsilon \} = o(n)
$$
for some $\pmb{a}=\{a_k\}_{k \in \Z} \in l^1$, and $\pmb{a}_j(n)$  given by \eqref{ajn}.  For any $\varphi \in C^\omega(D_{\mathcal{N}})$, we have
\begin{equation}
  \lim_{n \to \infty}  \sum_{k=1}^n \varphi(\lambda_k(A_n)) =  \frac{1}{2\pi} \int_{-\pi}^{\pi}  \varphi( a(t) )  \, dt
\end{equation}
with $a(t)= \sum_{k \in \Z} a_k e^{ikt}$. If the $A_n$'s are Hermitian, then the statement holds for any $\varphi \in C([-\mathcal{N},\mathcal{N}])$.
\end{cor}

\begin{proof}
This is a simple application of Theorem \ref{main normal}. Indeed, $\{A_n\} \in \mathcal{VMV}$ since the diagonals are given by density one convergent sequences and $\{A_n\}$ is $\delta_{\pmb{a}}$-distributed.
\end{proof}


\subsection{Illustrative examples}
\label{illustrativeexample}

Here we elaborate two examples of how our results apply to matrix sequences that are not of the \KMS type or other generalizations that have been previously considered. We only consider sequences of finite finite band size as our results can easily be extended to arbitrary sequences of matrices if one imposes conditions \eqref{cond szego} or \eqref{n-cond}.

\smallskip

Let $\{r_n\}$ be a sequence of positive numbers such that $r_n=o(n)$ and $r_n \to \infty$.  Also, let $\{c_n\}$ be a sequence that is $\nu$-distributed for some probability measure $\nu$ on $[0,1]$, i.e.
$$ \frac{1}{n} \sum_{k=1}^n \delta_{c_k} \to \nu.$$
We break up the diagonals into $\lfloor n/r_n\rfloor$ bins, where in each bin the entries tend toward constants.
To be precise, let $\{a(n)\}$ be a sequence for which $a_{k \lfloor r_n \rfloor +j} \to c_{k}$ as $j \to \infty$ with $k=0,...,\lfloor n/r_n \rfloor$ and $j=0,...,\lfloor r_n \rfloor$. For instance, construct the sequence $\{a(n)\}$ in the following way.  In the first $\lfloor r_n \rfloor$ entries of $a(n)$ place the constant $c_1$; in the next $\lfloor r_n \rfloor$ entries place $c_2$, and so on up to $a_{\lfloor r_n \rfloor}$, so that
\[a(n) = \{ \underbrace{c_1, c_1, \dots, c_1}_{\lfloor r_n \rfloor \mbox{ times}}, \, \underbrace{c_2, c_2, \dots, c_2}_{ \lfloor r_n \rfloor  \mbox{ times}} , \dots, \underbrace{c_{\lfloor r_n \rfloor}, c_{ \lfloor r_n \rfloor }, \dots, c_{\lfloor r_n \rfloor}}_{n-\lfloor n/r_n \rfloor  \, \lfloor r_n \rfloor \mbox{ times}} \}
\]
Now,  consider Riemann integrable functions $\hat{\alpha}_{-k_0},...,\hat{\alpha}_{k_0}$ on $[0,1]$ and construct the sequence of matrices $\{A_n\}$ by putting $\hat{\alpha}_k(a(n))$ on its $k$th diagonal. Then $\{A_n\}\in{\mathcal VMV}$ and is $\mu-$distributed with $\mu$ the push-forward measure of $\nu$ under the map $\alpha(s)=(\hat{\alpha}_{-k_0}(s),...,\hat{\alpha}_{k_0}(s))$ for $s \in [0,1]$.  Hence, we deduce
\begin{equation} \label{ex 1}
\lim_{n\rightarrow\infty}\frac{1}{n}  \sum_{k=1}^n \varphi(\lambda_k(A_n)) = \frac{1}{2\pi}\int_{-\pi}^{\pi} \int_0^1 \varphi(a(s,t) ) \, d\nu(s) \, dt
\end{equation}
for every $\varphi \in C^\omega(D_\M)$.

As a particular example, in the discrete Schr\"odinger case (i.e., $\hat{a}_1(s) = \hat{a}_{-1}(s)=1$, $\hat{a}_k(s)=0$ for $|k|>1$, and $f(s):=\hat{a}_0(s)$), \eqref{ex 1} reduces to
$$ \lim_{n\rightarrow\infty}\frac{1}{n}  \sum_{k=1}^n \varphi(\lambda_k(A_n))
= \frac{1}{2\pi}\int_{-\pi}^{\pi} \int_0^1 \varphi\left(f(x)+2\cos t \right) \, d\nu(x) \, dt
$$
One can make a simple change of variables in order to compute the asymptotic spectral density $\rho$ under some monotonicity assumptions. For instance, if we assume that $f$ is increasing on $[0,1]$ and $f(1)-f(0)<4$, then we can write \eqref{ex 1} as
$$   \lim_{n\rightarrow\infty}\frac{1}{n}  \tr[\varphi(A_n)]  =  \int_{f(0)-2}^{f(1)+2}  \varphi(x) \, \rho(x) \, dx$$
for any $\varphi \in C_c(\R)$ where
\[
\rho(x) = \left\{ \begin{array}{ll}
\dsp \int_{0}^{f^{-1}(x+2)} \frac{ d\nu(s)}{\sqrt{4-(x-f(s))^2}}  & x \in (f(0)-2,f(1)-2)\\
 \\
\dsp \int_{0}^1 \frac{d\nu(s)}{\sqrt{4-(x-f(s))^2}} & x \in (f(1)-2,f(0)+2) \\
\\
\dsp \int_{f^{-1}(x-2)}^1 \frac{d\nu(s)}{\sqrt{4-(x-f(s))^2}}  & x \in (f(0)+2,f(1)+2)
\end{array}\right.
\]


\smallskip

In our second example, we model each of the $\lfloor r_n\rfloor$ bins on the diagonals by Riemann integrable functions.
 Consider sequences of partitions $\{P_n\}$ with $P_n=\{t_{0;n},...,t_{r_n;n}\}$ of $[0,1]$ with $\text{mesh}(P_n)=o(1)$ and Riemann integrable functions $\hat{a}_{k;j}$ for $|k| \leq k_0$ and $j \in \N$. We construct the sequence of matrices $\{A_n\}$ for which the $k$th diagonal of $A_n$ is given by
\[ \underbrace{ \hat{a}_{k;1} (t_{0;n}), \ldots, \hat{a}_{k;1}(t_{r_n;n})}_{\lfloor r_n \rfloor \mbox{ times}}, \,  \underbrace{ \hat{a}_{k;2} (t_{0;n}), \ldots, \hat{a}_{k;2}(t_{r_n;n})}_{\lfloor r_n \rfloor \mbox{ times}}, \text{ etc}.
\]
The sequence $\{A_n\}$ is obviously in $\mathcal{VMV}$. Let $\pmb{\alpha}_j$ be the maps on $[0,1]$ defined by
\begin{equation*}
\pmb{ \alpha}_j (s) = (\hat{a}_{-k_0;j}(s),...,\hat{a}_{k_0;j}(s))  \qquad (j \in \N).
\end{equation*}
Under the assumption that the push-forward measures $n^{-1} \sum_1^n m_{\pmb{\alpha}_j}$ converge weakly to a measure $\mu$ on some compact set $X \subset \C^{2k_0+1}$, the sequence $\{A_n\}$ is $\mu$-distributed and hence Theorem \ref{main trace} implies
\begin{equation} \label{ex 2}
\lim_{n\rightarrow\infty}\frac{1}{n}  \sum_{k=1}^n \varphi(\lambda_k(A_n))= \frac{1}{2\pi}\int_{-\pi}^{\pi} \int_X \varphi(P(\pmb{z},t) ) \, d\mu(\pmb{z}) \, dt
\end{equation}
for any $\varphi \in C^\omega(D_\M)$.

 \section{Discussion}

In this paper we have restricted our attention to the scalar case.  Obvious extensions of  results on sequences of multi-level Toeplitz matrices or block Toeplitz matrices as considered in \cite{ti98b, SC03} will be explored in future works.  Here we mention a few other directions for future research.

It would be interesting to investigate the possible connections between our results and the discretization of differential operators. For instance, Jacobi matrices have been successfully used to study the discrete Schr\"odinger equation (see \cite{CySi87}).  This approach has also been used by the first author to compute the LSD of the quantum asymmetric top \cite{agbo09}.  Tilli~\cite{ti98b} used a similar approach in the study of Sturm-Liouville operators.  We note that our results should allow one to derive results about the spectra of differential operators whose coefficients are discontinuous.

Another possible direction to extend our work would be random Toeplitz matrices. Recently, Bryc et al. \cite{BrWo06} showed that  the LSD of random Toeplitz matrices exists, but were unable to provide a closed form for it. Kargin \cite{Ka09} slightly improved their results by looking at different asymptotic regimes, but he was still unable to explicitly compute the LSD for every regime. We believe that our methods can be used to compute the LSD of those matrices.

All the results presented in this paper are concerned with the First Szeg\H{o}'s Limit Theorem. Evidently, it would be of great interest to extend our results to the Strong Szeg\H{o}'s Limit Theorem  \cite{sz52},
which gives an explicit expression for the error term in the first theorem.  In particular, the strong theorem allows one to calculate the asymptotics of the determinant (as opposed to just the $n$th roots of the determinant).

Mejlbo and Schmidt \cite{mesc62} derived a Strong Szeg\H{o}'s Theorem for generalized Toeplitz matrices of the Kac-Murdock-Szeg\H{o} type under fairly restrictive conditions.
Later, Shao and Erhardt \cite{sh98, ehsh01} found a (different) expression for the error under more general conditions.
Both Mejlbo and Schmidt, and Shao and Erhardt require the functions $\hat{a}_k$ modeling the diagonals to be H\"older continuous with exponent $\alpha \geq 3/2$.  Our results for the trace only require Riemann integrability, so it may be possible to extend their results to a broader class of operators.


\begin{thebibliography}{10}

\bibitem{agbo09}
Alfonso~F. Agnew and Alain Bourget.
\newblock A trace formula for a family of {J}acobi operators.
\newblock {\em Anal. Appl. (Singap.)}, 7(2):115--130, 2009.

\bibitem{bo12}
Alain Bourget.
\newblock Spectral density of {J}acobi matrices with small deviations.
\newblock {\em Constr. Approx.}, 36(3):375--398, 2012.

\bibitem{BrWo06}
W{\l}odzimierz Bryc, Amir Dembo, and Tiefeng Jiang.
\newblock Spectral measure of large random {H}ankel, {M}arkov and {T}oeplitz
  matrices.
\newblock {\em Ann. Probab.}, 34(1):1--38, 2006.

\bibitem{SC03}
S.~Serra Capizzano.
\newblock Generalized locally toeplitz sequences; spectral analysis and
  applications to discretezed differential equations.
\newblock {\em Linear Algebra Appl.}, 366:371--402, 2003.

\bibitem{CySi87}
H.~L. Cycon, R.~G. Froese, W.~Kirsch, and B.~Simon.
\newblock {\em Schr\"odinger operators with application to quantum mechanics
  and global geometry}.
\newblock Texts and Monographs in Physics. Springer-Verlag, Berlin, study
  edition, 1987.

\bibitem{ehsh01}
Torsten Ehrhardt and Bin Shao.
\newblock Asymptotic behavior of variable-coefficient {T}oeplitz determinants.
\newblock {\em J. Fourier Anal. Appl.}, 7(1):71--92, 2001.

\bibitem{grsz58}
Ulf Grenander and Gabor Szeg{\H{o}}.
\newblock {\em Toeplitz forms and their applications}.
\newblock California Monographs in Mathematical Sciences. University of
  California Press, Berkeley-Los Angeles, 1958.

\bibitem{hj85}
Roger~A. Horn and Charles~R. Johnson.
\newblock {\em Matrix analysis}.
\newblock Cambridge University Press, Cambridge, 1990.
\newblock Corrected reprint of the 1985 original.

\bibitem{kamusz53}
M.~Kac, W.~L. Murdock, and G.~Szeg{\"o}.
\newblock On the eigenvalues of certain {H}ermitian forms.
\newblock {\em J. Rational Mech. Anal.}, 2:767--800, 1953.

\bibitem{Ka09}
Vladislav Kargin.
\newblock Spectrum of random {T}oeplitz matrices with band structure.
\newblock {\em Electron. Commun. Probab.}, 14:412--421, 2009.

\bibitem{mesc62}
Lars~C. Mejlbo and Palle~F. Schmidt.
\newblock On the eigenvalues of generalized {T}oeplitz matrices.
\newblock {\em Math. Scand.}, 10:5--16, 1962.

\bibitem{Ne79}
Paul~G. Nevai.
\newblock Orthogonal polynomials.
\newblock {\em Mem. Amer. Math. Soc.}, 18(213):v+185, 1979.

\bibitem{qi84}
Li~Qun Qi.
\newblock Some simple estimates for singular values of a matrix.
\newblock {\em Linear Algebra Appl.}, 56:105--119, 1984.

\bibitem{sh98}
Bin Shao.
\newblock A trace formula for variable-coefficient {T}oeplitz matrices with
  symbols of bounded variation.
\newblock {\em J. Math. Anal. Appl.}, 222(2):505--546, 1998.

\bibitem{si09}
Barry Simon.
\newblock Regularity and the {C}es\`aro-{N}evai class.
\newblock {\em J. Approx. Theory}, 156(2):142--153, 2009.

\bibitem{sz52}
G\`abor Szeg\H{o}.
\newblock On certain hermitian forms associated with fourier series of a
  positive function.
\newblock {\em Festschrift Marcel Riesz, Lund.}, pages 228--238, 1952.

\bibitem{ti98b}
P.~Tilli.
\newblock Locally {T}oeplitz sequences: spectral properties and applications.
\newblock {\em Linear Algebra Appl.}, 278(1-3):91--120, 1998.

\bibitem{TyZa98}
E.~E. Tyrtyshnikov and N.~L. Zamarashkin.
\newblock Spectra of multilevel {T}oeplitz matrices: advanced theory via simple
  matrix relationships.
\newblock {\em Linear Algebra Appl.}, 270:15--27, 1998.

\end{thebibliography}

\end{document}